\documentclass{article}
\usepackage{graphicx}
\usepackage[latin1]{inputenc}
\usepackage{amsmath,amsthm,amssymb,amsfonts}
\usepackage{color}           % Farbunterst"utzung

\newtheorem{theorem}{Theorem}[section]
\newtheorem{lemma}[theorem]{Lemma}

\newtheorem{definition}[theorem]{Definition}
\newtheorem{proposition}[theorem]{Proposition}
\newtheorem{remark}[theorem]{Remark}
\newtheorem{example}[theorem]{Example}

\numberwithin{equation}{section}

\newcommand{\average}{{\mathchoice {\kern1ex\vcenter{\hrule
height.4pt width 8pt depth0pt}
\kern-11pt} {\kern1ex\vcenter{\hrule height.4pt width 4.3pt
depth0pt} \kern-7pt} {} {} }}

\newcommand{\weakst}{\stackrel{\ast}{\rightharpoonup}}

\newcommand{\N}          {\mathbb{N}}

\newcommand{\R}         {\mathbb{R}}

\def\debaixodasetafraca#1#2{\mathrel{}\mathop{\rightharpoonup}\limits^{#1}_{#2}}

\begin{document}

\begin{center}
  \LARGE
Differential Inclusions and $\mathcal A$-quasiconvexity
\\[3mm]
\normalsize
Ana Cristina Barroso$^\dag$,
Jos\'e Matias$^\ddag$ and Pedro Miguel Santos$^\ddag$
\\[3mm]
{\em 
$\dag$
Faculdade de Ci\^encias da Universidade de Lisboa\\
Departamento de Matem\'atica and CMAF-CIO\\
Campo Grande, Edif\'\i cio C6, Piso 1, 1749-016 Lisboa, Portugal\\
$\ddag$ Departamento de Matem\'atica and CAMGSD\\
Instituto Superior T\'ecnico\\
Av. Rovisco Pais,
1049-001 Lisboa,
Portugal\\
[3mm]}
%\\
%Email:  abarroso@ptmat.fc.ul.pt}
%${}^\ddag$
 \today\\[3mm]
%\begin{minipage}{0.8\textwidth}
%\end{minipage}
\end{center}

\medskip

\begin{abstract}
In this paper we extend to the abstract $\cal A$-framework some existence theorems for differential inclusion problems 
with Dirichlet boundary conditions.
\end{abstract}

\section{Introduction}

There is an extensive literature on differential inclusions to treat problems of the type
\begin{equation}\label{gradiente}
\left\{\begin{array}{ll} Du(x) \in E, \,\,\,\,\,\text{for a.e.}\,x \in \Omega\\
u(x)=\varphi(x),\,\,\text{for}\,x\in \partial \Omega\\
\end{array}
\right.
\end{equation}
where $\Omega \subset \R^n$ is a bounded, open set, $E \subset \R^{m.n}$ is a bounded set and 
$\varphi \in W^{1,\infty}(\Omega;\R^m)$. 
We refer, in particular, {\it the method of convex integration} introduced by Gromov \cite{Gr_73}, which was applied by 
M\"{u}ller and Sverak  \cite{Mu_Sv_96} to solve the problem of two potential wells in two dimensions, and also 
{\it the Baire category method} introduced by Cellina \cite{Ce_80}, which we use in this paper, following essentially the 
framework considered by Dacorogna and Marcellini in \cite{Dac_Mar_99}.

The objective of this work is to generalise the above problem to the $\cal A$-quasiconvexity framework, that is, 
we replace the condition $$Du(x) \in E$$ by the more general condition $$v(x) \in E,\qquad {\cal A}v=0,$$
where $\cal A$ is a first order linear partial differential operator (see the next section for details). 
In particular if $\cal A={\rm curl}$ we recover the gradient case (\ref{gradiente}).

Following ideas from \cite{Dac_Mar_99}, we consider sets $E$ of the form
$$E=\{ \xi \in \R^d: F_i (\xi)=0, i=1,..,I \},$$
where $F_i$, $i=1,..,I$ are continuous, $\cal A$-quasiconvex functions. The lower semicontinuity of integrals of 
$\cal A$-quasiconvex functions with respect to the $L^{\infty}$-weak star convergence of $\cal A$-free sequences 
plays a central role in the verification of the Baire theorem hypothesis, which we need in order to obtain our 
existence result.
In Dacorogna and Pisante \cite{Dac_Pis_05} more general sets $E$ were considered, thus removing the hypothesis of being 
the intersection of the sets of zeros of quasiconvex functions. In the proof they use an idea of Kirchheim \cite{K_15} 
about the existence of a dense set of continuity points for the operator $D$ (see \cite{Dac_Pis_05} for details).
However, it is not clear how to generalise that proof to the $\cal A$-quasiconvex case.

In this paper we start with an abstract existence result, under constant boundary conditions (see Theorem \ref{aet}). 
This result is of very difficult application, since it depends on the verification of a property called the relaxation 
property (see Definition \ref{repo}). 
We then identify some particular cases where the relaxation property can be verified, and we obtain an existence theorem,
which is a particular case of Theorem \ref{aet}. This is the main result of this paper, and it was obtained for the one 
level set case (see Theorem \ref{di1}), that is, 
$$E = \{ \xi \in \R^d: F(\xi)=0 \},$$
with $F$ $\cal A$-quasiconvex and coercive. We focus mainly on the 
two-dimensional case (although extensions to higher dimensions are discussed in Remark \ref{3D}), and we impose some 
restrictions on the operator $\cal A$, namely a constant (maximum) rank condition, conditions on the dimensions and the 
kernel of the matrices $A^{(1)}$ and $A^{(2)}$ (see the following sections for notation and details). Some of these 
conditions can be relaxed using essentially the same proofs.

In the case of several level sets, that is, 
$$E=\{ \xi \in \R^d: F_i (\xi)=0, i=1,..,I \},$$
where $F_i$, $i=1,..,I$, are continuous, $\cal A$-quasiconvex functions, we obtain a sufficient condition for the 
laminate convex hull $K=\Lambda{\rm co}E$ to have the relaxation property with respect to $E$ (see Theorem \ref{aproxprop}). 
As an application, we obtain a second existence theorem (see Theorem \ref{system}), which is quite general, in the sense that 
it includes several level sets and no restrictions on the operator $\cal A$ are made, but it requires the laminate convex 
hull of $E$ to be compact and strongly star shaped. 

In both existence results described above we use the laminate convex hull, that is, the set $\Lambda{\rm co}E$, 
which is easier to obtain then the closed $\cal A$-quasiconvex hull, the natural set for this kind of problems. 
In particular, we prove that $\Lambda{\rm co}E$ can be obtained departing from $E$ by successively adding segments in 
the directions of the characteristic cone (see Proposition \ref{induction} for details), and this is used in the proof 
of Theorem \ref{aproxprop}. 

The outline of this work is as follows: in Section 2 we fix the notation and recall some basic notions related to 
$\cal A$-quasiconvexity, in Section 3 we state the problem and present our results.

\section{Preliminaries}
\subsection{Notation}

The following notation will be used throughout:
\begin{itemize}
\item[-] $\Omega$ will denote a bounded, open subset of $\R^N$;
\item[-] $Q$ denotes the unit cube of $\R^N$;
\item[-] $\mathcal{L}^N$ denotes the $N$-dimensional Lebesgue measure;
\item[-] we write a.e. in $\Omega$ meaning $\mathcal{L}^N$ a.e. in $\Omega$;
\item[-] given $K \subset \R^{d}$, we denote by $\text{int}\,K$ its interior, which is understood to be non-empty;
\item [-] given $E \subset \R^{d}, \; \text{dist}(. ; E)\,$ stands for the euclidean distance to $E$ in $\R^d$;
\item [-] ${\mathcal D}'(\Omega; \R^d)$ denotes the space of distributions in $\Omega$ with values in $\R^d$;
\item [-] $L_{\rm piec}^\infty(\R^N; \R^d)$ denotes the set of piecewise constant and bounded functions from $\R^N$ to 
$\R^d$;
\item [-] given a differential operator $\cal A$, we denote by 
${\cal N({\cal A})} := \{u: { \cal A} u = 0\; \text{in} \; {\mathcal D}'\}$.
\end{itemize} 

\subsection{{$\mathcal A$}-quasiconvexity}

Here we deal with operators $\mathcal{A} : {\mathcal D}'(\Omega; \R^d) \to {\mathcal D}'(\Omega; \R^m)$ of the form
$${\cal A}u := \sum_{i=1}^N A^{(i)} \frac {\partial u} {\partial x_i}$$
where $A^{(i)}$, $i=1,...,N$, are $m \times d$ matrices. Operators like {\it curl} or {\it div} have the above form, 
precisely,
\begin{enumerate}
\item  (${\cal A}={\rm div}$) for $u \in {\mathcal D}'(\Omega; \R^d)$, with $d=N$, 
$${\cal A} u = \sum_{i=1}^N \frac {\partial u^i} {\partial x_i};$$
\item (${\cal A}={\rm curl}$) for  $u \in {\mathcal D}'(\Omega; \R^d)$, with $d=m \times N$, 
$${\cal A} u = {\left( \frac {\partial u_k^j} {\partial x_i} -  
\frac{\partial u_i^j}{\partial x_k} \right)}_{j=1,..,m; i,k=1,..,N};$$
\item (Maxwell's Equations)  for  $u \in {\mathcal D}'(\R^3; \R^{3 \times 3})$, 
$$ {\cal A} u = \left( {\rm div} (m+h), {\rm curl} \,h \right),$$
where $u = (m,h).$
\end{enumerate}

We present some definitions related to the theory of $\cal A$-quasiconvexity which will be useful throughout the paper.

\begin{definition} ($\cal A$-quasiconvex function)
Let $F:\R^d \to \R$ be a Borel measurable function. $F$ is $\cal A$-quasiconvex if and only if the inequality below holds
$$F(\xi) \leq \int_Q F(\xi +w(x))\,dx$$
for every $w \in C^{\infty}(\R^N;\R^d)$, $Q$-periodic, verifying the condition ${\cal A} w=0$ in $\R^N$.
\end{definition}

\begin{definition}\label{cone} 
The characteristic cone of ${\cal A}$ is defined by 
$$\Lambda=\left\{v\in \R^d: \exists w\in \R^N\setminus \{0\},\,\, \left(\sum_{i=1}^{N} A^{(i)} w_i\right)v=0 \right\}.$$
\end{definition}

\smallskip
 
\noindent  For each  $\lambda \in \Lambda$ we define the subspace of $\R^N$
$$\mathcal{V}_\lambda := \left\{ \omega \in \R^N:  \left(\sum_{i=1}^{N} A^{(i)} w_i\right)\lambda=0 \right\},$$
which will be useful in the sequel.
 
\begin{remark} \label{nopenazsaltos}
{\rm We note that if $\lambda \in \Lambda$  and $u(x) = \lambda v(x\cdot w^1, x\cdot w^2,\ldots, x\cdot w^l)$, where 
$w^1,\ldots,w^l$ is an orthonormal basis for $\mathcal{V}_\lambda$, then 
$${\cal A}u=0 \; \text{in}\; \mathcal{D}'(\R^N; \R^m).$$
In particular, the equation does not penalise jumps of $\lambda$ (or multiples) in the directions of 
$\mathcal{V}_{\lambda}$.
In fact, let $a, b \in \R^d$ with $b-a \in \Lambda$ and $\nu \in \mathcal{V}_{b-a}$ and consider $u: \R^N \to \R^d$ 
defined by
$$ u (x) = \left\{ \begin{array}{ll} b &\text{if} \; \; x \cdot \nu > 0\\
&\\
a & \text{if} \; \;  x \cdot \nu < 0.
\end{array}
\right.
$$
Let $\phi \in C^\infty_c ( \R^N; \R^m)$ and denote by $S: = \{ x \in \R^N : x \cdot \nu = 0\}.$ Then, 
\begin{eqnarray*} 
\langle \mathcal{A} u, \phi \rangle &=& 
\left\langle \sum_{i=1}^N  A^{(i)} \frac{\partial u}{\partial x_i}, \phi \right \rangle = - \sum_{i=1}^N  \left \langle A^{(i)} u, \frac{\partial \phi}{ \partial x_i} \right \rangle \\\\
&&  = - \sum_{i=1}^N  \left( \int_{\R^N \setminus S}  \frac{\partial (A^{(i)} u \cdot \phi)}{\partial x_i}\, dx 
- \int_{\R^N \setminus S}  \frac{\partial (A^{(i)} u)}{\partial x_i}\phi \, dx\right) 
= - \sum_{i=1}^N  \int_{\R^N \setminus S}  \frac{\partial( A^{(i)} u \cdot \phi)}{\partial x_i}\, dx.
\end{eqnarray*}
Therefore, 
$$  \langle \mathcal{A} u, \phi \rangle = - \sum_{i=1}^N  \int_{S} A^{(i)}(b-a)\phi \nu^i \, dx = 0,$$
\noindent since
$$ \left( \sum_{i=1}^N A^{(i)} \nu^i \right)(b-a)= 0. $$}
\end{remark}

\begin{example} 
{\rm An explicit characterisation of the characteristic cone of ${\mathcal A}$ is provided in the following examples.
\begin{enumerate}
\item  If ${\cal A}={\rm div}$, then ${\Lambda}=\R^d$ and ${\mathcal V}_\lambda$ is the orthogonal complement of $\lambda$.
\item If ${\cal A}={\rm curl}$, then 
$$\Lambda= \left\{ \lambda \in \R^{m \times N}: \lambda_i^j = a_j \nu_i \,\, \text{for some} \; a \in \R^m \;
\text{and} \; \nu \in \R^N \right\}.$$
\item If ${\cal A}$ is the operator for Maxwell's Equations, one has
$$\Lambda= \left\{ \lambda=(m,h) \in \R^{3 \times 3}: m = a- h, \text{ for some} \; a \in \R^3 \; \text{such that} 
\; a\cdot h=0 \right\}.$$
\end{enumerate} }
\end{example}

\section{Differential Inclusions in the ${\cal A}$-framework}

Let
$\mathcal{A} : {\mathcal D}'(\Omega; \R^d) \to {\mathcal D}'(\Omega; \R^m)$ be a first order linear differential 
operator of constant maximum rank, that is 
$${\rm rank} \left( \sum_{i=1}^N A^{(i)} w_i  \right) = m,$$
for every $w \neq 0$. Let $E \subset \R^d$, $\xi \in \R^d$. 
We look for necessary and sufficient conditions (over $E$) such that the following problem (P) attains solutions:
\begin{equation*}
(P) \; \; \left\{\begin{array}{ll}
u(x) \in E \,\,\; {\rm for}\,  {\rm a.e.}\; x \in \Omega,\\\\
u \in L^\infty(\R^N; \R^d), \;  {\mathcal A}u = 0 \; ( {\rm in}\; {\mathcal D}')\\\\
u = \xi \; {\rm on}\; \R^N \setminus \Omega.\\
\end{array}
\right.
\end{equation*}

\begin{definition}\label{repo} (Relaxation Property)
Let $K, E \subset \R^d$. We say that $K$ has the relaxation property with respect to $E$ if
\begin{eqnarray*} 
\forall \, \Omega, \forall \, \xi \in {\rm int} K & \exists \{u_n\} 
\subset L_{\rm piec}^\infty(\R^N; \R^d),& \; u_n \debaixodasetafraca{*}{} \xi, \; u_n = \xi \; {\rm in}\; 
\R^N\backslash \Omega,\\\\
&& \mathcal{A}u_n = 0 \; {\rm in}\; {\mathcal D}'(\R^N;\R^m)\\\\
&& u_n(x) \in E \cup {\rm int} K\; {\rm a.e.} \; x \in \Omega,\\\\
&&\int_\Omega {\rm dist}(u_n(x); E)\; dx \to 0, \; ( n \to \infty).
\end{eqnarray*}
\end{definition}

\begin{remark}\label{Rbound}
{\rm In fact it is enough to verify the above definition for a cube, as for a general open set $\Omega$ we can use 
Vitali's covering theorem. 

Indeed, suppose that the relaxation property holds for some particular cube $Q$, denote by $u_n$ the corresponding 
sequence and let $R:= {\rm sup}_n {||u_n||}_{\infty}$. Let $\Omega$ be an arbitrary open, bounded set and consider 
a disjoint covering of $\Omega$ with closed cubes $Q^j$, $j \in \N$, with side length small enough such that 
$Q^j \subset \Omega$ for every $j \in \N$ and ${\cal L}^N (\Omega \setminus \cup_j Q^j)=0.$  
As the weak star topology of $L^\infty$ is metrisable in bounded sets we can consider a distance, denoted by 
$d_{\star}^R$. 

The sequence $u_n$ can be rescaled for each cube $Q^j$, and we denote by $u_n^j$ this rescaled sequence. Now we 
construct a new sequence $v_m$ defined by
$$ v_m= \left\{ \begin{array}{ll} u_{n_m}^j &{\rm if}\; x \in Q^j, j=1,..,m\\
&\\
\xi & {\rm otherwise\,\, in}\,\,\R^N,\\
\end{array}
\right.
$$
where $n_m$ is a subsequence of $n$ such that the two conditions below hold
\begin{equation*}
d_{\star}^R ( u_{n_m}^j; \xi) < \frac1{m^2},\,\,\,\,\int_{Q^j} {\rm dist} (u_{n_m}^j; E) \,dx < \frac1 {m^2},
\end{equation*}
for every $j=1,..,m$. Clearly the sequence $v_m$ verifies all the conditions in the definition of the relaxation property 
for the set $\Omega$. We note also that $v_m=\xi$ in a neighbourhood of $\partial \Omega$ (with``size" changing obviously
with $m$), this observation will be useful in the proof of Theorem \ref{aet} below.  }
\end{remark}

We start with an abstract existence theorem.

\begin{theorem}\label{aet}
Let $\Omega \subset \R^N$ be an open set and let $F : \R^d \to \R$ be a continuous and $\cal A$-quasiconvex function. 
Let $E,K \subset \R^d$ be such that
\begin{eqnarray*}
E = \{ \xi \in \R^d: F(\xi)=0 \},\,\,\,\, K \subset \{ \xi \in \R^d: F(\xi) \leq 0 \}.
\end{eqnarray*}
Assume that $K$ has the relaxation property with respect to $E$, and assume that $E$ and $K$ are both bounded.
Then if $\xi \in E \cup int K$ there exists  (a dense set of)  $u \in L^{\infty} (\R^N;\R^d) \cap {\cal N({\cal A})}$ 
such that
$$F(u(x))=0 \,\,{\text {for a.e.}}\,\,x \in \Omega; \,\,u=\xi\,\,{\text in}\,\, \R^N \setminus \Omega.$$
\end{theorem}
\begin{proof}
The proof follows ideas of \cite{Dac_Mar_99}. 

Let
$$V= \{ u \in L^{\infty}_{\rm piec} (\R^N;\R^d): \mathcal A u=0,\,  u=\xi \, 
\text{in} \, \R^N \setminus \Omega, u(x) \in E \,\cup\, {\rm int} K \}.$$
Note that $V$ is nonempty since $u\equiv \xi \in$ V.

Let ${\overline V}$ be the closure of $V$ with respect to the weak* topology in $L^{\infty}$, note that as $E$ and $K$ 
are bounded we can find a metric for the weak* topology, which will be denoted in the sequel by $d$. We introduce the set
$$V_{\varepsilon} = \left\{ u \in {\overline V}: \int_{\Omega} F(u(x))\; dx  > - \varepsilon \right\}.$$

If $V_{\varepsilon}$ is open and dense in ${\overline V}$, then by  Baire's category theorem 
$\cap_{\varepsilon >0} V_{\varepsilon}$ is nonempty, and thus the statement of the theorem holds.

The fact that $V_{\varepsilon}$ is open in ${\overline V}$ (or equivalently that ${\overline V}\backslash V_{\varepsilon}$ 
is closed in ${\overline V}$) follows immediately from Theorem 6.3 in \cite{Fon_Mul_99} which ensures that, if $F$ is
continuous and $\mathcal A$-quasiconvex, then
$$ \{u_n\} \subset L^\infty , \; u_n \debaixodasetafraca{*}{} u, \; \; \mathcal{A}(u_n) = 0 \;  
\Rightarrow \liminf \int_\Omega F(u_n(x))\, dx \geq \int_\Omega F(u(x))\, dx.$$

Now we prove the density of $V_{\varepsilon}$.
Let $u \in V$ and $\delta > 0$. We want to find an element $u_\delta \in V_\varepsilon$ such that 
$$d (u,u_\delta) < \delta.$$
As $u$ is piecewise constant in $\Omega$  we can find open sets $A_i$ such that $\cup_{i=1}A_i = \Omega$, 
${\cal L}^N \left(\Omega \setminus \cup_{i=1}A_i\right)=0$ and
$u(x)=\xi_i$ for $x \in A_i$. Suppose that $\{ i: \xi_i \in \text{int} K\} \neq \emptyset $ 
(if $ \forall i, \; \xi_i \in E $ there is nothing to prove). 
Then as $K$ has the relaxation property with respect to $E$, we can find a sequence $u_n^{\xi_i}$ of 
piecewise constant functions defined in $\R^N$ such that ${\cal A}u_n^{\xi_i} =0$, 
$u_n^{\xi_i} (x) \in E \cup {\rm int} K$, 
$u_n^{\xi_i} \debaixodasetafraca{*}{} \xi_i$, $u_n^{\xi_i}=\xi_i$ in $\R^N \setminus A_i$  and 
$$\int_{A_i} {\rm dist}(u_n^{\xi_i};E) \,dx \to 0.$$
Define the sequence $u_n(x)=u_n^{\xi_i}$ for $x \in A_i$ and $u_n(x) = \xi$ in $\R^N \setminus \Omega$.
It is clear from Remark \ref{Rbound} that the sequences coming from the relaxation property can be constructed in such 
a way that they are constant in a neighbourhood of each set $A_i$, and hence we have ${\cal A}u_n=0$. 
Thus $u_n \in V$, for all $n$, and $u_n \weakst u$.

Let $\displaystyle 0 < \varepsilon'< \frac{\varepsilon}{2\mathcal{L}^N(\Omega)}$. 
Since $F$ is uniformly continuous in $\overline{K}$ and $F=0$ on $E$, we can find 
$\delta = \delta(\varepsilon')$ such that
$$ {\rm dist}(u_n;E) \leq \delta \Rightarrow  F(u_n(x)) \geq -\varepsilon'.$$
As $\displaystyle \int_{A_i} {\rm dist}(u_n^{\xi_i};E) \,dx \to 0$ we conclude that 
$\displaystyle \int_\Omega {\rm dist}(u_n;E) \, dx \to 0.$ This strong convergence in $L^1$ yields, by
equi-integrability, $\displaystyle \lim_{n} \mathcal{L}^N\{ x \in \Omega: {\rm dist}(u_n(x);E) > \delta\} = 0.$
Since 
$$A_{n,\varepsilon'}:= \{ x \in \Omega: F(u_n(x)) < -\varepsilon'\} \subset 
\{x \in \Omega: {\rm dist}(u_n;E) >  \delta\},$$ 
and as $u_n \weakst u$  and hence is uniformly bounded in $L^\infty$,
we have by the continuity of $F$,
\begin{eqnarray*} 
\int_{\Omega} F(u_n(x))\, dx &=& \sum_{i}\int_{A_i \cap A_{n,\varepsilon}}F(u_n(x))\, dx +
\sum_{i}\int_{A_i \setminus A_{n,\varepsilon'}}F(u_n(x))\, dx \\
&\geq& -C \mathcal{L}^N(A_{n,\varepsilon'}) - \varepsilon' \mathcal{L}^N(\Omega)\\
&\geq& -\frac{\varepsilon}{2} -\frac{\varepsilon}{2} = -\varepsilon,
\end{eqnarray*}
for $n$ large enough. Hence $u_n \in V_{\varepsilon}$.
Moreover, as $u_n \weakst u$, we have 
$$d(u;u_n) < \delta,$$
also for $n$ large enough.
\end{proof}

\begin{remark} \label{Elimitado} 
{\rm (Possible extensions of the existence theorem)

\noindent 1) The above theorem also holds for sets $E$ of the form
$$E= \{ \xi \in \R^d: F_i (\xi)=0, i=1,...,I \}$$
where $F_i$, $i=1,..,I$, are continuous, $\cal A$-quasiconvex functions. The proof is similar to the case of a single function.

\noindent 2) Piecewise constant boundary conditions. }
\end{remark}

The main difficulty to apply the above theorem is the verification of the relaxation property. In the following lemma we 
present a (laminate) construction that will be key in establishing an existence result for the case of one level set 
($I=1$ in Remark \ref{Elimitado}) under coercivity conditions on $F$ in at least one direction of the characteristic 
cone (see Theorem \ref{di1}).

\begin{lemma}\label{lemma1} Let $N=2$, $d = 2m$, and suppose that 
${\cal N} \left( A^{(1)} \right) \cap {\cal N} \left( A^{(2)} \right)= \{ 0 \}$.
Let $a,b \in \R^d$ be such that $b-a$  belongs to the {\it characteristic cone} and $$\lambda a + (1-\lambda)b=0.$$ 
Then we can find a sequence of piecewise constant functions  ${\tilde u}_n: \R^2 \to \R^d$ such that 
${\tilde u}_n \debaixodasetafraca{*}{} 0$ and
$$\mathcal{L}^N\left( \{ x \in Q: u_n(x)=a \} \right) \to \lambda,\,\,\,\,\mathcal{L}^N\left( \{ x \in Q: u_n(x)=b \} \right) \to 1-\lambda,$$
${\cal A}{\tilde u}_n=0$, ${\tilde u}_n = 0$ in $\R^2 \setminus Q$. Moreover, given $\varepsilon >0$, we have for $n$ 
large enough $${\rm Range} (u_n) \setminus \{a,b\} \subset B_{\varepsilon}(0).$$
\end{lemma}

\begin{proof}
Let $a,b \in \R^d$ and suppose without loss of generality that $$A^{(1)} (b-a)=0,$$ (that is, $e_1 \in {\cal V}_{b-a}$), 
otherwise we rotate the cube $Q$.
Note that from $\lambda a + (1-\lambda) b=0$ it follows that $A^{(1)}a=A^{(1)}b=0$.
For $n \in \N$, consider the equations
\begin{equation}\label{505} 
c_n \in {\cal N}(A^{(2)}) \cap \left( a + {\cal N} \left( \frac{A^{(1)}}{\sqrt n}+\frac {A^{(2)}}{n}\right)\right).
\end{equation}
We prove that, for each $n$, equation \eqref{505} admits a solution. Let $n=1$, and rewrite \eqref{505} as
\begin{equation}\label{508}
M_1 c_1 = \left( 0, \left( A^{(1)}+A^{(2)} \right)a \right),
\end{equation}
where $M_1$ is a $2m \times d$ matrix defined by lines as $M_1:= \left(A^{(2)}, A^{(1)}+ A^{(2)} \right)$. Note that 
as ${\rm rank} (M_1)=2m$ then the columns of $M_1$ generate $\R^{2m}$ and thus \eqref{508} has a unique solution since 
$d=2m$. For $n>1$ the argument is similar.

Now consider the periodic function $u_n$, defined in the strip $Q \cap \left\{ 0 < x_1 < \frac1 {n} \right\}$ as follows
$$ u_n(x_1,x_2)= \left\{ \begin{array}{ll} a 
&{\rm if}\; x_1 < \frac{\lambda}{n}, \,\, {\sqrt n} x_1 < x_2  < 1-{\sqrt n}x_1,\\
&\\
c_n & {\rm if}\; x_1 < \frac{\lambda}{n}, \,\,  1-{\sqrt n}x_1 < x_2 < 1\,\\
&\\
-c_n & {\rm if}\; x_1 < \frac {\lambda}{n},\,\, 0 < x_2 < {\sqrt n} x_1\,\\
&\\
b & {\rm if}\; x_1 > \frac{\lambda}{n}, \,\, {\sqrt n}\frac{\lambda}{\lambda-1} \left(x_1 - \frac {\lambda}{n}\right) 
+ \frac {\lambda}{\sqrt n} < x_2 < {\sqrt n} \frac{\lambda}{1-\lambda} \left(x_1 - \frac {\lambda}{n} \right) 
+ 1 - \frac{\lambda}{\sqrt n}\,\\
&\\
c_n & {\rm if}\; x_1 > \frac {\lambda}{n}, \,\,  {\sqrt n}\frac{\lambda}{1-\lambda} \left(x_1 
- \frac {\lambda}{n} \right) + 1-\frac {\lambda}{\sqrt n} < x_2 < 1\,\\
&\\
-c_n & {\rm if}\; x_1 > \frac {\lambda}{n}, \,\, 0 < x_2 < {\sqrt n} \frac{\lambda}{\lambda-1} \left(x_1 
- \frac{\lambda}{n} \right) + \frac {\lambda}{\sqrt n}\,\\
\end{array}
\right.
$$
and then extended to the whole unit cube $Q$ by periodicity in the $x_1$-direction.

Note that $${\cal A}u_n = 0\,\,\,\,\text{ in}\,\, Q.$$ 
Indeed if, for instance, we consider the line 
$x_2 = {\sqrt n} \frac{\lambda}{\lambda-1} \left(x_1 - \frac{\lambda}{n} \right) + \frac {\lambda}{\sqrt n}$, 
whose normal is $\nu=\left(  {\sqrt n} \frac{\lambda}{\lambda-1},-1\right)$, where the function jumps $b+c_n$, 
we have $\left( A^{(1)}{\sqrt n} \frac{\lambda}{\lambda-1} - A^{(2)}\right) (b+c_n)=0$ (see Remark \ref{nopenazsaltos}). 
Moreover, if we extend $u_n$ by $0$ to all of $\R^2$, we obtain a sequence ${\tilde u}_n$ such that
$${\cal A}{\tilde u}_n =0\,\,\text{ in}\,\,\R^2.$$
In addition, $c_n \to 0$. Indeed, as $A^{(2)} c_n=0$ and $\left( \frac{A^{(1)}}{\sqrt n}+\frac {A^{(2)}}{n}\right) (c_n-a)=0$, 
it follows that $$A^{(1)} c_n =  \frac {A^{(2)}}{{\sqrt n}} a.$$ Thus we can choose $c_n = \frac {c_1} {{\sqrt n}}$.
\end{proof}

Some of the conditions of the previous lemma, namely the dimensional condition ($N=2$) and the restrictions on the operator 
$\cal A$ ($d=2m$ and ${\cal N} \left( A^{(1)} \right) \cap {\cal N} \left( A^{(2)} \right)= \{ 0 \}$) can be improved. 
We refer to Remark \ref{3D} for a discussion of possible extensions to dimension $N=3$. 
However, as the proofs are essentially the same, in what follows we restrict ourselves to the hypotheses stated in the 
lemma, for simplicity of presentation. 

We present examples of systems that verify the conditions of the above lemma.

\begin{example}
{\rm (curl, N=2)
In this case $d=2m$ and we consider $u=(u_i^j)(=(u_1^1,u_2^1,...))$, $i=1,2$ and $j=1,...,m$ (that is, $u$ is the 
derivative of a function $v=(v^1,...,v^m)$), we then have the system of m equations
$$\frac {\partial u_2^j} {\partial x_1}-\frac {\partial u_1^j} {\partial x_2}=0.$$
Here $A^{(1)}=(e_2, e_4,...)$ and $A^{(2)}=(-e_1,-e_3,...)$, where $e_1, e_2,\ldots$ are the unit coordinate vectors 
of $\R^d$.}
\end{example}

\begin{example}
{\rm (N=2, d=4)
Consider the system of two equations
$$\frac {\partial u_1} {\partial x_1} + \frac {\partial u_2} {\partial x_2}=0,\,\,\,\,\, -\frac {\partial u_4} 
{\partial x_1}+\frac {\partial u_3} {\partial x_2}=0,$$
thus, by lines, we have $A^{(1)}=(e_1,-e_4)$ and $A^{(2)}=(e_2,e_3)$. In this case  $\Lambda \subsetneq \R^4$, 
precisely,
$$\Lambda= \{ (\lambda_1,\lambda_2, \lambda_3, \lambda_4) \in \R^4: \lambda_1\lambda_3+\lambda_2\lambda_4=0 \}.$$}
\end{example}

\begin{remark}\label{3D} 
{\rm We now give an outline of the extension of the previous lemma to the case $N=3$, $d=3m$ and 
$\displaystyle \cap_{i=1}^3 {\cal N}(A^{(i)})= \{0\}.$
Let $a, b \in \R^3$ be such that 
$$A^{(1)} (b-a)=0,$$ (that is, $e_1 \in {\cal V}_{b-a}$), and 
$$A^{(3)} (b-a)=0,$$ 
(that is, $e_3 \in {\cal V}_{b-a}$), otherwise we continuously deform the cube $Q$.
Note that from $\lambda a + (1-\lambda) b=0$ it follows that $A^{(1)}a=A^{(1)}b=0, $ and $A^{(3)}a=A^{(3)}b=0$.

For $x_3$ fixed, $ \frac{1}{n} < x_3 < 1 - \frac{1}{n}$, consider the construction in the previous lemma, which
holds provided that 
$$c \in {\cal N}(A^{(2)}) \cap \left( a + {\cal N} \left(A^{(1)}+ A^{(2)}\right)\right).$$
We want to extend this construction up to the faces of the cube with normal $\pm e_3$. 
In order to do this we need that
$$c \in {\cal N}(A^{(2)})\cap {\cal N}(A^{(3)}) \cap 
\left( a + {\cal N} \left(A^{(1)}+ A^{(2)}\right)\right).$$
Then the reasoning in the proof for the case $N=2$ can be followed in order to conclude the result.}
\end{remark}

We next present some definitions that will be needed for our one level set existence theorem (cf. Theorem \ref{di1}).

\begin{definition} Let $F: \R^d \to \R\cup\{+\infty\}$. We say that $F$ is $\Lambda$-convex if it is convex 
along the directions of the characteristic cone $\Lambda$.
\end{definition}

\begin{definition}\label{Lambdaconv}
Given $E\subset \R^d$ we define
$$ \Lambda{\rm co}E := \{ \xi \in \R^d: F(\xi)\leq 0, \forall F: \R^d \to \R\cup\{+\infty\}: F\lfloor E = 0, 
F\; {\rm is}\; \Lambda-\text{convex}\}.$$
\end{definition}

\begin{theorem}\label{rp1}
Let $N=2$, $d = 2m$, and suppose that 
${\cal N} \left( A^{(1)} \right) \cap {\cal N} \left( A^{(2)} \right)= \{ 0 \}$.
Let $F: \R^d \to \R$ be continuous, $\Lambda$-convex and suppose that $F$ is coercive, i.e.,
$\displaystyle \lim_{|\xi| \to +\infty} F(\xi)=+\infty$. Let 
$$E = \{ \xi : F(\xi) = 0\}.$$ Then  $\Lambda{\rm co}E$ has the relaxation property with respect to $E$.
\end{theorem}
\begin{proof}
The proof follows that of Theorem $6.11$ in \cite{Dac_Mar_99}.

First note that $\Lambda{\rm co}E = \{ \xi : F(\xi) \leq 0\}.$
Clearly we have that $\Lambda {\rm co}E \subset \{ \xi : F(\xi) \leq 0\}.$
On the other hand, assuming that $F(\xi_0) < 0$ ( if $F(\xi_0) = 0$ there is nothing to prove), 
choose a direction $\lambda \in \Lambda$. By coercivity we have that there exist $t_1 < 0 < t_2$ such that
$$ F(\xi_0 + t\lambda) < 0, \forall t \in (t_1, t_2),$$ 
\noindent and $F(\xi_0 + t_1\lambda) = F(\xi_0 + t_2\lambda) = 0$ and hence $\xi_0 + t_i\lambda \in E, i=1,2.$ 
Given $G$ admissible for the definition of $\Lambda {\rm co}E$, we have that $G(\xi_0 + t_i\lambda) = 0, i=1,2$ and since  
$G$ is $\Lambda$-convex we conclude that $G(\xi_0) \leq 0$, that 
is, $\xi_0 \in \Lambda {\rm co}E$.

We now prove the relaxation property. We need to show that  for every $\xi \in{\rm int} \Lambda{\rm co}E$ 
we can construct a sequence as follows
\begin{eqnarray*}   
(u_n) \subset L_{\rm piec}^\infty(\R^N; \R^d),&& \; u_n \debaixodasetafraca{*}{} \xi, \; u_n = \xi \; {\rm in}
\; \R^N\backslash Q, \; \mathcal{A}u_n = 0 \; {\rm in}\; {\mathcal D}'(\R^N;\R^m)\\\\
&\,& \,u_n(x) \in E \cup {\rm int} \Lambda {\rm co}E\; {\rm a.e.} \; x \in \Omega,\\\\
&\,&\int_\Omega {\rm dist}(u_n(x); E)\; dx \to 0, \; ( n \to \infty).
\end{eqnarray*}

\noindent  By the coercivity of $F$ and choosing $\lambda \in \Lambda$ we have that 
$$\exists t_1 < 0 < t_2: F(\xi + t\lambda) < 0, \; \forall t \in (t_1, t_2), $$
$$ F(\xi + t_1\lambda ) = F(\xi + t_2\lambda) = 0.$$
Set $$a:= \xi+t_1 \lambda, \; b:= \xi+t_2  \lambda.$$
Notice that $b - a = (t_2 - t_1)\lambda \in \Lambda $ and that for $\eta := \frac {t_2} {t_2-t_1}$ we have that 
$\xi = \eta a + (1 - \eta) b$.
Hence the result follows immediately from Lemma \ref{lemma1} (which obviously applies also to the case $\xi \neq 0$). 
\end{proof}

By combining Theorems \ref{aet} and \ref{rp1}, and taking into account the construction in Lemma \ref{lemma1},  
we have the following result.

\begin{theorem}\label{di1}
Let $N=2$, $d = 2m$, and suppose that ${\cal N} \left( A^{(1)} \right) \cap {\cal N} \left( A^{(2)} \right)= \{ 0 \}$. 
Let $F$ be continuous, $\mathcal{A}$-quasiconvex, suppose that $F$ is coercive, i.e.,
$\displaystyle \lim_{|\xi| \to +\infty} F(\xi)=+\infty$ and set
$$E= \{ \xi: F(\xi)=0 \}.$$
Let $\xi \in E \cup {\rm int} \Lambda{\rm co}E$. Then $(P)$ has a solution.
\end{theorem}

\begin{proof}
Follows from Theorems \ref{aet} and \ref{rp1}.
\end{proof}

\begin{example}
{\rm The problem 
$$u \in L^{\infty}(\R^2;\R^2), \; u_1^2+u_2^2=1 \; {\rm in} \; \Omega, {\rm div}\,u=0, \; 
u=\xi \; {\rm in} \; \R^2 \setminus \Omega$$ 
has infinitely many solutions if $|\xi| < 1$. Indeed, it is enough 
to apply Theorem \ref{di1} to the function $F(x_1,x_2)=x_1^2+x_2^2-1$.}
\end{example}

We now examine the several level sets case. Examples presented in \cite{Dac_Mar_99}, in the context of gradients, 
show that it is important to consider more general sets $E$, namely to treat the case of 
$$E= \{ \xi \in \R^d: F_i (\xi)=0, i=1,...,I \}$$
with $F_i$, $i=1,..,I$, continuous, $\cal A$-quasiconvex functions. With this extension in mind we present 
more general results about the relaxation property.
We start with an alternative characterisation of $\Lambda{\rm co}E$ (see Definition \ref{Lambdaconv}).
This set is also called by some authors, in the case of $\mathcal{A} = \;{\rm curl}$, the {\it laminate} 
$\Lambda$-convex hull and is, in general, different from the $\Lambda$-convex hull defined using only real-valued 
functions.

\begin{proposition}\label{induction}
Setting $\Lambda_0{\rm co}E = E$ 
\noindent and
$$ \Lambda_{i+1}{\rm co}E = \{ \xi \in \R^d: \xi = ta + (1-t)b, \; a, b \in \Lambda_{i}{\rm co}E, \, 
b-a \in \Lambda, \; t \in [0,1]\},$$
\noindent then
$$ \Lambda{\rm co}E = \cup_{i \in \N} \Lambda_{i}{\rm co}E.$$
\end{proposition}
\begin{proof}
An easy induction argument shows that 
$\Lambda_{i}{\rm co}E \subset \Lambda{\rm co}E$, $\forall i \in \N$ so that
$\cup_{i \in \N} \Lambda_{i}{\rm co}E \subset \Lambda{\rm co}E.$ 
In order to prove the reverse inclusion, note that given a function 
$f: \R^d \to \R\cup\{+\infty\}$ its $\Lambda$-convex  envelope is defined by
$$ \Lambda f = \; {\rm sup}\; \{ g: g \; {\rm is} \; \Lambda-{\rm convex},\; g\leq f \}.$$

We want the following characterisation to hold (similar to the one by Kohn-Strang, cf. \cite{KS}, in the rank-one
case).

Setting $\Lambda_0 f = f$ and for $i \in \N$
$$ \Lambda_{i+1} f(\xi) = \inf \{ t\Lambda_i f (a) + ( 1- t) \Lambda_i f(b), \; t \in [0,1], \; \xi = ta + (1-t)b, 
\; b-a \in \Lambda\},$$
\noindent then we claim that
$$ \Lambda f (\xi) = \inf_{i \in \N} \Lambda_i f(\xi).$$

Notice that for this characterisation to make sense, we need to be able to express a general vector $\xi \in \R^d$ by a 
convex combination of two vectors $ a, b \in \R^d$ such that $b-a \in \Lambda$. 
This can be done in the following way: given $\xi \in \R^d$ and $c, d \in \R^d$ such that $ c- d \in \Lambda$, we write
$ \xi = ta + (1-t)b$, for 
$$ a = \xi - (1-t)(c-d), \; \; b = \xi + t(c-d).$$
To prove the claim it suffices to note the following:
\begin{itemize}
\item[i)] the sequence $\Lambda_i f$ is decreasing, i.e. $\Lambda_{i+1} f \leq \Lambda_i f$, and
equality holds only if $\Lambda_i f$ is $\Lambda$-convex;
\item[ii)] assuming that $f$ is bounded from below then $\Lambda_i f$ is also bounded from below
(notice that we will apply this characterisation to a characteristic function of a set so this hypothesis is not restrictive in our case).
\end{itemize}
From i) and ii) it follows that the sequence $\Lambda_i f$ converges. We denote by
$$ \varphi(\xi) = \lim_{i \to + \infty}\Lambda_i f(\xi) = \inf_{i \in \N}\Lambda_i f(\xi).$$
Clearly, $\varphi \leq f$, and it is easy to see that $\varphi$ is $\Lambda$-convex, and that
if $g \leq f$ is $\Lambda$-convex then $g \leq \varphi$. We conclude that $\varphi = \Lambda f$.

We now apply the above characterisation to $\chi_E$, the characteristic function of the set $E$ (characteristic function 
in the convex analysis sense). An induction argument shows that
$$ \Lambda_i \chi_E = \chi_{\Lambda_i{\rm co} E}$$
\noindent and thus, passing to the limit,
$$ \Lambda \chi_E = \chi_ {\cup_{i} \Lambda_i {\rm co}E}.$$  
Since $\Lambda \chi_E \lfloor E = 0$ and it is a $\Lambda$-convex function we deduce that for 
$\xi \in \Lambda {\rm co}E$ we have that $\Lambda \chi_E (\xi) = 0$ so the previous equality
yields $\xi \in \cup_{i \in \N} \Lambda_{i}{\rm co}E$ and the proof is complete.
\end{proof}

\begin{theorem} \label{aproxprop} Let $E \subset \R^d$ be a compact set, and suppose that there exists a family 
$E_{\delta}$ of sets, with $\delta >0$, with the property that for every $\varepsilon >0$ there exists a 
$\delta_0=\delta_0(\varepsilon)$ such that
\begin{equation}\label{aprox}
E_{\delta} \subset \{ \xi: {\rm dist} (\xi;E) \leq \varepsilon \}
\end{equation}
for every $\delta \in (0,\delta_0]$. Suppose also that
\begin{equation}\label{aprox1}
K(E_{\delta}) = {\Lambda{\rm co} E_\delta} \subset {\rm int} \Lambda{\rm co} E,
\end{equation}
and that for every $\xi \in {\rm int} \Lambda{\rm co} E$ we have $\xi \in K(E_{\delta})$ for $\delta$ small enough.
Then $ \Lambda {\rm co} E$ has the relaxation property with respect to $E$.
\end{theorem}
\begin{proof}
Fix $k \in \N$. Let $\xi \in  {\rm int} \Lambda{\rm co} E$, then $\xi \in \Lambda{\rm co} E_{\delta_k}$ for some 
$\delta_k$ small enough which we assume to verify the condition $\delta_k < \frac1 {k}$. By Proposition \ref{induction} 
there exists $I_k \in \N$ such that $\xi \in \Lambda{\rm co}^{I_k} E_{\delta_k}$. Thus we have
$$\xi=\lambda \xi_1 + (1 - \lambda) \xi_2$$
for some $\lambda \in [0,1]$,  and $\xi_1$, $\xi_2 \in \Lambda{\rm co}^{I_k-1} E_{\delta_k}$ such that 
$\xi_2 - \xi_1 \in \Lambda$. 
Applying Lemma \ref{lemma1} we can find a 
sequence of piecewise constant functions $u_n^k \weakst \xi$ as $n \to \infty$ , such that 
$u_n^k (x) \in {\rm int} \Lambda{\rm co} E$ for a.e. $x \in Q$, ${\cal A}u_n^k=0$, moreover
$$\mathcal{L}^N\left(\{ x: u_n^k (x) \notin \{ \xi_1, \xi_2 \} \}\right) \to 0$$
as $n \to \infty$.
Applying  Lemma \ref{lemma1} successively, we find a new sequence $v_n^k$ of piecewise constant functions, such that
\begin{equation}\label{m0}
\mathcal{L}^N\left(\{ x: v_n^k (x) \notin E_{\delta_k} \}\right) \to 0
\end{equation}
as $n \to \infty$. We also have 
$$v_n^k \weakst \xi,\,\,\,\,\, {{\cal A} v_n^k}=0,\,\,\,\,\, \int_{Q} {\rm dist}(v_n^k;E_{\delta_k}) \to 0$$
as $n \to \infty$, where the last condition follows from (\ref{m0}) and the fact that $\Lambda{\rm co} E$ is bounded.
Now it is enough to consider an appropriate sequence $u_k = v_{n_k}^k$, and use (\ref{aprox}).
\end{proof}

\begin{remark}
{\rm Conditions (\ref{aprox}) and (\ref{aprox1}) are a particular case of the approximation property in \cite{Dac_Mar_99} 
(see Definition 6.12).}
\end{remark}

\begin{remark}
{\rm In particular, for ${\cal A}={\rm curl}$, the above theorem is similar to Theorem 6.14 in \cite{Dac_Mar_99}.
In \cite{Dac_Mar_99}, the general abstract Theorem 6.14 is also refined to sets of the form
$$E_{\delta}:= \{ \xi: F_i^{\delta}(\xi)=0, \, i=1,..,I \}$$
where $F_i^{\delta}, i=1,..,I$, is a family of quasiconvex functions, continuous with respect to the parameter 
$\delta \in [0, \delta_0]$, and $F_i^0=F_i$  (see Theorem 6.22),
and it is used to obtain existence results for differential inclusions involving singular values (Theorem 7.28) 
or potential wells (Theorem 8.5).}
\end{remark}

If we are able to compute $\Lambda{\rm co} E$, and if it is compact and star shaped, the following theorem gives an 
existence result. However, we note that even when $E$ is compact, we do not necessarily have that 
$\Lambda{\rm co}E$ is compact, even in the case of $\mathcal{A}$ = curl.

\begin{theorem} \label {system}
Let $\Omega \subset \R^N$ be an open set, let $F_i : \R^d \to \R$ be continuous and 
$\mathcal A$-quasiconvex, $i = 1, \ldots, I$, and let
$$E = \left\{\xi \in \R^d : F_i(\xi) = 0, i = 1, \ldots, I\right\}.$$
Assume that $\Lambda {\rm co}E$ is compact and strongly star shaped with respect to a fixed
$\xi_0 \in {\rm int}\Lambda {\rm co}E$ (i.e. for every 
$\xi \in \Lambda {\rm co}E$ and every $t \in (0,1]$ we have 
$t\xi_0 + (1-t) \xi \in {\rm int}\Lambda {\rm co}E$). If $0 \in E \cup {\rm int}\Lambda {\rm co}E$
then there exists (a dense set of) $u \in L^{\infty}(\Omega;\R^d)$ such that
$F_i(u(x)) = 0$ for a.e. $x \in \Omega$, $i = 1, \ldots, I$, $u = 0$ on $\partial \Omega$.
\end{theorem}
\begin{proof}
This is a consequence of the generalisation of Theorem \ref{aet} (see also Remark \ref{Elimitado}) when we set 
$K = \Lambda {\rm co}E$,
provided we show that $\Lambda {\rm co}E$ has the relaxation property. This in turn will follow from 
Theorem \ref{aproxprop} and from the fact that $\Lambda {\rm co}E$ is strongly star shaped. 

Indeed, for $\delta \in (0,1]$ let
$$E_\delta = \delta \xi_0 + (1 - \delta)E.$$
An induction argument shows that 
$$K(E_\delta) = \Lambda{\rm co}E_\delta = \delta \xi_0 + (1-\delta)\Lambda {\rm co}E
\subseteq {\rm int}\Lambda {\rm co}E, \forall \delta \in (0,1].$$
\end{proof}

\bigskip
\textbf{Acknowledgments.} 
The research of A.C.B. was partially supported by the Funda\c{c}\~{a}o para a Ci\^{e}ncia e a Tecnologia through grant 
UID/MAT/04561/2013. The research of J.M. and P.S. was partially supported by the Funda\c{c}\~{a}o
para a Ci\^{e}ncia e a Tecnologia through grant UID/MAT/04459/2013.

\end{document}